\documentclass[12pt,reqno]{amsart}
\usepackage{latexsym,amsmath,amsfonts,amssymb,amsthm}
\theoremstyle{plain}
\newtheorem{Thm}{Theorem}
\newtheorem{Lem}{Lemma}
\newtheorem{Cor}{Corollary}

\theoremstyle{definition}
\newtheorem*{Ack}{Acknowledgment}
\theoremstyle{remark}
\newtheorem*{Rem}{Remark}
\def\Z{\mathbb Z}
\def\N{\mathbb N}

\def\Q{\mathbb Q}

\def\1{{\bf 1}}
\def\pmod #1{\ ({\rm mod}\ #1)}

\def\ord{{\rm ord}}
\def\floor #1{\lfloor{#1}\rfloor}
\def\biggfloor #1{\bigg\lfloor{#1}\bigg\rfloor}

\def\qbinom #1#2#3{{\genfrac{[}{]}{0pt}{}{#1}{#2}}_{#3}}

\begin{document}
\title{Factors of Some Lacunary $q$-Binomial Sums}
\author{Hao Pan}
\email{haopan79@yahoo.com.cn}
\address{Department of Mathematics, Nanjing University,
Nanjing 210093, People's Republic of China}
\thanks{The author is supported by National Natural Science Foundation of
China (Grant No. 10901078).}
\begin{abstract}
In this paper, we prove a divisibility result for the lacunary $q$-binomial sum
$$
\sum_{k\equiv r\pmod{c}}(-1)^kq^{\binom{k}{2}}\qbinom{n}{k}{q}\qbinom{(k-r)/c}{l}{q^{c}}.
$$
\end{abstract}
\maketitle

\section{Introduction}
\setcounter{equation}{0}
\setcounter{Thm}{0}
\setcounter{Lem}{0}
\setcounter{Cor}{0}

Suppose that $p$ is a prime. A classical result of Fleck asserts that
\begin{equation}\label{fleck}
\sum_{k\equiv r\pmod{p}}(-1)^k\binom{n}{k}\equiv
0\pmod{p^{\lfloor\frac{n-1}{p-1}\rfloor}},
\end{equation}
where $\floor{x}=\max\{z\in\Z:\, z\leq x\}$ is the floor function. In 1977, Weisman generalized Fleck's congruence to prime power moduli in the following way:
\begin{equation}\label{weisman}
\sum_{k\equiv r\pmod{p^\alpha}}(-1)^k\binom{n}{k}\equiv
0\pmod{p^{\lfloor\frac{n-p^{\alpha-1}}{p^{\alpha-1}(p-1)}\rfloor}}.
\end{equation}
In 2009, with help of $\psi$-operator in Fontaine's theory of
$(\phi,\Gamma)$-modules, Sun \cite{S} and Wan \cite{W} obtained a polynomial-type extension of (\ref{fleck}) and (\ref{weisman}):
\begin{equation}\label{sw}
\sum_{k\equiv r\pmod{p^\alpha}}(-1)^k\binom{n}{k}\binom{(k-r)/p^\alpha}{l}\equiv
0\pmod{p^{\lfloor\frac{n-p^{\alpha-1}-lp^\alpha}{p^{\alpha-1}(p-1)}\rfloor}}.
\end{equation}

On the other hand, motivated by the homotopy exponents of the special
unitary group ${\rm SU}(n)$, Davis and Sun \cite{DS,SD} proved another two congruences with a little different flavor:
\begin{equation}e
\label{ds}
\sum_{k\equiv r\pmod{p^{\alpha}}}(-1)^k\binom{n}{k}((k-r)/p^\alpha)^l
\equiv0\pmod{p^{\nu_p(\floor{n/p^{\alpha}}!}},
\end{equation}
\begin{equation}
\label{sd}
\sum_{k\equiv r\pmod{p^{\alpha}}}(-1)^k\binom{n}{k}\binom{(k-r)/p^\alpha}{l}
\equiv0\pmod{p^{\nu_p(\floor{n/p^{\alpha-1}}!)-l-\nu_p(l!)}},
\end{equation}
where $\nu_p(x)=\max\{i\in\N: p^i\mid x\}$ is the $p$-adic order of $x$. Notice that
neither (\ref{ds}) nor (\ref{sd}) could be deduced from (\ref{sw}), though (\ref{ds}) and (\ref{sd}) are often weaker than (\ref{sw}) provided $l$ is small.

In this paper, we shall consider the $q$-analogues of (\ref{ds}) and (\ref{sd}). For an integer $n$, as usual, define the $q$-integer
$$
[n]_q=\frac{1-q^n}{1-q}.
$$
And define the $q$-binomial coefficient
$$
\qbinom{n}{k}q=\frac{[n]_q\cdot[n-1]_q\cdots[n-k+1]_q}{[k]_q\cdot[k-1]_q\cdots[1]_q}.
$$
In particular, we set $\qbinom{n}0q=1$ and $\qbinom{n}kq=0$ for $k<0$. It is easy to see $\qbinom{n}{k}q$ is a polynomial in $q$ since
$$
\qbinom{n+1}{k}q=q^k\qbinom{n}{k}q+\qbinom{n}{k-1}q.
$$
Let $\Z[q]$ denote the polynomial ring in $q$ with integral coefficients. Then we have the following $q$-analogue of (\ref{sd}).
\begin{Thm}\label{qT1} For $n,c\in\Z^+$ and $r,h\in\Z$, the lacunary $q$-binomial sum
$$
\sum_{k\equiv r\pmod{c}}(-1)^kq^{\binom{k}{2}+hk}\qbinom{n}{k}{q}\qbinom{(k-r)/c}{l}{q^{c}}
$$
is divisible by
$$
\prod_{c\mid d}\Phi_d(q)^{\floor{n/d}-\floor{lc/d}}\cdot\prod_{\substack{b\mid c\\ b<c}}\Phi_b(q)^{\floor{n/b}-\floor{r/b}-\floor{(n-r)/b}}
$$
over $\Z[q]$, where $\Phi_d$ is the $d$-th cyclotomic polynomial.
\end{Thm}
Since $\Phi_{p^\alpha}(q)=[p]_{q^{p^{\alpha-1}}}$ for prime $p$, we may get
\begin{align}\label{qc1}
&\sum_{k\equiv r\pmod{p^{\alpha}}}(-1)^kq^{\binom{k}{2}+hk}\qbinom{n}{k}{q}\qbinom{(k-r)/p^\alpha}{l}{q^{p^\alpha}}\notag\\
\equiv&0\pmod{\prod_{j=\alpha}^\infty[p]_{q^{p^{j-1}}}^{\floor{n/p^j}-\floor{l/p^{j-\alpha}}}\cdot\prod_{j=1}^{\alpha-1}[p]_{q^{p^{j-1}}}^{\floor{n/p^j}-\floor{r/p^j}-\floor{(n-r)/p^j}}
}.
\end{align}
Note that
\begin{equation}\label{ndd}
\nu_p(n!)=\sum_{j=1}^{\infty}\bigg\lfloor\frac{n}{p^j}\bigg\rfloor,
\end{equation}
and for $1\leq j\leq\alpha-1$
$$\bigg\lfloor\frac{n}{p^j}\bigg\rfloor-\bigg\lfloor\frac{r}{p^j}\bigg\rfloor-\bigg\lfloor\frac{n-r}{p^j}\bigg\rfloor
=
\bigg\lfloor\frac{\{r\}_{p^{\alpha-1}}+\{n-r\}_{p^{\alpha-1}}}{p^j}\bigg\rfloor-\bigg\lfloor\frac{\{r\}_{p^{\alpha-1}}}{p^j}\bigg\rfloor
-\bigg\lfloor\frac{\{n-r\}_{p^{\alpha-1}}}{p^j}\bigg\rfloor,$$
where $\{r\}_{p^{\alpha-1}}$ denotes the least non-negative residue of $r$ modulo $p^{\alpha-1}$.
Substituting $q=1$ in (\ref{qc1}), we can get the following stronger version of (\ref{sd}) \cite[(1.1)]{SD}:
\begin{align}
\label{sdt}
&\nu_p\bigg(\sum_{k\equiv r\pmod{p^{\alpha}}}(-1)^k\binom{n}{k}\binom{(k-r)/p^\alpha}{l}\bigg)\notag\\
\geqslant&\nu_p(\floor{n/p^{\alpha-1}}!)-l-\nu_p(l!)+\tau_{p}(\{r\}_{p^{\alpha-1}},\{n-r\}_{p^{\alpha-1}}),
\end{align}
where
$$
\tau_p(a,b)=\ord_p\Big(\binom{a+b}{a}\Big).
$$

We shall prove Theorem \ref{qT1} in the next section. For the advantage of $q$-congruences, our proof of Theorem \ref{qT1} is even simpler than the original one of (\ref{sd}).
\begin{Rem}
Quite recently, some Fleck type $q$-congruences also have been established by Schultz and Walker \cite{SW}.
\end{Rem}

\section{Proofs of Theorem \ref{qT1}}
\setcounter{equation}{0}
\setcounter{Thm}{0}
\setcounter{Lem}{0}
\setcounter{Cor}{0}
Let $\Q[q]$ denote the polynomial ring in $q$ with rational coefficients.
Note that the greatest common divisor of all coefficients of $\Phi_d(q)$ is $1$. By a well-known result of Gauss, if $\Phi_d(q)$ divides $F(q)\in\Z[q]$ over $\Q[q]$, then $\Phi_d(q)$ also divides $F(q)$ over $\Z[q]$. So below we don't distinguish the $q$-congruences over $\Z[q]$ and $\Q[q]$.
\begin{Lem}
\label{L1}
$(\zeta^rq^h;q)_n$ is divisible by
$\Phi_{d}(q)^{\floor{n/d}}$ for any $r, s\in\Z$, where $\zeta=e^{2\pi\sqrt{-1}/d}$.
\end{Lem}
\begin{proof}
We know that
$$
\Phi_{d}(q)=\prod_{\substack{k=1\\ (k,d)=1}}^{d}(1-\zeta^k q).
$$ For any $k$ with $(k,d)=1$, let
$0\leq e_k<d$ be the integer such that $e_kk\equiv r\pmod{d}$.
Then we have
$
1-\zeta^{r}\zeta^{-e_kk}=0$,
i.e., $1-\zeta^k q$ divides $1-\zeta^{r}q^{j}$ if
$j\equiv e_k\pmod{d}$. Thus
$$
(\zeta^rq^h;q)_n=\prod_{j=h}^{n+h-1}(1-\zeta^rq^j)
$$
is divisible by $(1-\zeta^k q)^{\floor{n/d}}$.
\end{proof}
\begin{Lem}
\label{L2}
\begin{equation}
\label{L2e}
\sum_{k\equiv r\pmod{c}}
(-1)^kq^{\binom{k}{2}+hk}\qbinom{n}{k}{q}\equiv
0\pmod{\prod_{c\mid d}\Phi_{d}(q)^{\floor{n/d}}}.
\end{equation}
\end{Lem}
\begin{proof}
In view of the $q$-binomial theorem (cf. \cite[Corollary 10.2.2(c)]{AAR}),
$$
\sum_{k=0}^n(-1)^kq^{\binom{k}{2}}\qbinom{n}{k}{q}x^k=(x;q)_n.
$$
So letting $\zeta=e^{2\pi\sqrt{-1}/c}$,
\begin{align*}
\sum_{k\equiv r\pmod{c}}
(-1)^kq^{\binom{k}{2}+hk}\qbinom{n}{k}{q}=&\frac{1}{c}\sum_{k=0}^n
(-1)^kq^{\binom{k}{2}+hk}\qbinom{n}{k}{q}\sum_{t=0}^{c-1}\zeta^{(k-r)t}\\
=&\frac{1}{c}\sum_{t=0}^{c-1}\zeta^{-rt}(\zeta^tq^h;q)_n.
\end{align*}
Thus (\ref{L2e}) immediately follows from Lemma \ref{L1}, since $\zeta$ is also a $d$-th root of unity if $c\mid d$.
\end{proof}
\begin{Lem}
\label{L3}
$$
\qbinom{n}{k}{q}=\prod_{1<d\leqslant n}\Phi_d(q)^{\floor{n/d}-\floor{k/d}-\floor{(n-k)/d}}.
$$
\end{Lem}
\begin{proof}
Clearly
\begin{equation}\label{nd}
[n]_q!=\prod_{j=1}^n[j]_q=\prod_{j=1}^n\prod_{\substack{d>1\\ d\mid j}}\Phi_d(q)=\prod_{1<d\leqslant n}\Phi_d(q)^{\floor{n/d}}.
\end{equation}
Hence
$$
\qbinom{n}{k}{q}=\frac{[n]_q!}{[k]_![n-k]_q!}=\prod_{1<d\leqslant n}\Phi_d(q)^{\floor{n/d}-\floor{k/d}-\floor{(n-k)/d}}.
$$
\end{proof}
\begin{proof}[Proof of Theorem \ref{qT1}]
We shall prove
\begin{equation}
\label{T1e}
\sum_{k\equiv r\pmod{c}}
(-1)^kq^{\binom{k}{2}}\qbinom{n}{k}{q}\qbinom{(k-r)/c}{l}{q^{c}}
\equiv 0\pmod{\prod_{c\mid d}\Phi_d(q)^{\floor{n/d}-\floor{lc/d}}}
\end{equation}
by using an induction on $l$. The case $l=0$ follows from Lemma \ref{L2}. Assume that $l\geqslant 1$ and (\ref{T1e}) holds for the smaller values of $l$. Compute
\begin{align*}
&\sum_{k\equiv r\pmod{c}}
(-1)^kq^{\binom{k}{2}+hk}\qbinom{n}{k}{q}\qbinom{(k-r)/c}{l}{q^{c}}\\
=&\sum_{k\equiv r\pmod{c}}
(-1)^kq^{\binom{k}{2}+hk}\qbinom{n}{k}{q}\cdot\frac{q^{-r}([k]_q-[r]_q)}{[c]_q[l]_{q^{c}}}\cdot\qbinom{(k-r)/c-1}{l-1}{q^{c}}\\
=&\frac{q^{-r}[n]_q}{[lc]_q}\sum_{k\equiv r\pmod{c}}
(-1)^kq^{\binom{k}{2}+hk}\qbinom{n-1}{k-1}{q}\qbinom{(k-r-c)/c}{l-1}{q^{c}}\\
&-\frac{q^{-r}[r]_q}{[lc]_q}\sum_{k\equiv r\pmod{c}}
(-1)^kq^{\binom{k}{2}+hk}\qbinom{n}{k}{q}\qbinom{(k-r-c)/c}{l-1}{q^{c}}.
\end{align*}
Note that
$[lc]_{q}$ is divisible by or prime to $\Phi_{d}(q)$
according to whether $d\mid lc$ or not. By the induction hypothesis, we obtain that
\begin{align*}
&\frac{q^{-r}[n]_q}{[lc]_q}\sum_{k\equiv r\pmod{c}}
(-1)^kq^{\binom{k}{2}+hk}\qbinom{n-1}{k-1}{q}\qbinom{(k-r-c)/c}{l-1}{q^{c}}\\
=&\frac{q^{s-r}[n]_q}{[lc]_q}\sum_{k\equiv r-1\pmod{c}}
(-1)^kq^{\binom{k}{2}+(h+1)k}\qbinom{n-1}{k}{q}\qbinom{(k-(r-1+c))/c}{l-1}{q^{c}}\\
\equiv&0\pmod{\prod_{c\mid d}\Phi_{d}(q)^{\floor{(n-1)/d}+\1_{d\mid n}-\floor{(l-1)c/d}-\1_{d\mid lc}}},
\end{align*}
and
\begin{align*}
&\frac{q^{-r}[r]_q}{[lc]_q}\sum_{k\equiv r\pmod{c}}
(-1)^kq^{\binom{k}{2}+hk}\qbinom{n}{k}{q}\qbinom{(k-(r+c))/c}{l-1}{q^{c}}\\
\equiv&0\pmod{\prod_{c\mid d}\Phi_{d}(q)^{\floor{n/d}-\floor{(l-1)c/d}-\1_{d\mid lc}}},
\end{align*}
where for an assertion $A$ we adopt the notation
$
\1_{A}=1$ or $0$ according to whether $A$ holds or not.
Thus by noting that for arbitrary positive integers $s$ and $t$
$$
\1_{t\mid s}=\biggfloor{\frac{s}{t}}-\biggfloor{\frac{s-1}{t}},
$$
(\ref{T1e}) is concluded.

On the other hand, with help of Lemma \ref{L3},
$$
\qbinom{n}{k}{q}\equiv 0\pmod{\prod_{b\mid c}\Phi_{b}(q)^{\floor{n/b}-\floor{r/b}-\floor{(n-r)/b}}}
$$
whenever $k\equiv r\pmod{c}$, since for any $b\mid c$
$$
\biggfloor{\frac{n}{b}}-\biggfloor{\frac{k}{b}}-\biggfloor{\frac{n-k}{b}}=\biggfloor{\frac{n}{b}}-\biggfloor{\frac{r}{b}}-\biggfloor{\frac{n-r}{b}}.
$$
All are done.\end{proof}
It is easy to check that
$$
[l]_{q^c}[l-1]_{q^c}\cdots[2]_{q^c}[1]_{q^c}=\prod_{1<j\leq l}\Phi_j(q^c)^{\floor{l/j}}\equiv0\pmod{\prod_{\substack{c\mid d\\ d>c}}\Phi_d(q)^{\floor{lc/d}}}
$$
and
$$
[k]_q[k-1]_q\cdots[k-l+1]_q=q^{-\binom{l}{2}}[k]_q([k]_q-[1]_q)\cdot([k]_q-[l-1]_q).
$$
So applying a simple induction on $l$, we can deduce the $q$-analogue of (\ref{ds}):
\begin{Cor}\label{c1}
$$
\sum_{k\equiv r\pmod{c}}(-1)^kq^{\binom{k}{2}+hk}\qbinom{n}{k}{q}[(k-r)/c]_{q^c}^l.
$$
is divisible by
$$
\Phi_c(q)^{\floor{n/c}-l}\prod_{\substack{c\mid d\\ d>c}}\Phi_d(q)^{\floor{n/d}}\cdot\prod_{\substack{b\mid c\\ b<c}}\Phi_b(q)^{\floor{n/b}-\floor{r/b}-\floor{(n-r)/b}}.
$$
\end{Cor}
In particular, for prime $p$,
$$
\sum_{k\equiv r\pmod{p^{\alpha}}}(-1)^kq^{\binom{k}{2}+hk}\qbinom{n}{k}{q}[(k-r)/p^\alpha]_q^l
$$
is divisible by
$$
[p]_{q^{p^{\alpha-1}}}^{\floor{n/p^\alpha}-l}\prod_{j=\alpha+1}^\infty[p]_{q^{p^{j-1}}}^{\floor{n/p^j}}\cdot\prod_{j=1}^{\alpha-1}[p]_{q^{p^{j-1}}}^{\floor{n/p^j}-\floor{r/p^j}-\floor{(n-r)/p^j}}.
$$
Substituting $q=1$, we obtain the improvement of (\ref{ds}) \cite[Theorem 5.1]{DS}:
\begin{align}
\label{dst}
&\nu_p\bigg(\sum_{k\equiv r\pmod{p^{\alpha}}}(-1)^k\binom{n}{k}((k-r)/p^\alpha)^l\bigg)\notag\\
\geqslant&\max\{\nu_p(\floor{n/p^{\alpha-1}}!)-l,\nu_p(\floor{n/p^{\alpha}}!)\}+\tau_{p}(\{r\}_{p^{\alpha-1}},\{n-r\}_{p^{\alpha-1}}).
\end{align}

\section{Lucas type and Wolstenholme-Ljunggren type $q$-congruences}
\setcounter{equation}{0}
\setcounter{Thm}{0}
\setcounter{Lem}{0}
\setcounter{Cor}{0}

Let
$$
T_{p^\alpha,l}(n,r)=\frac{l!p^l}{\floor{n/p^{\alpha-1}}!}\sum_{k\equiv r\pmod{p^{\alpha}}}
(-1)^k\binom{n}{k}\binom{(k-r)/p^\alpha}{l}.
$$
In \cite{SD}, Sun and Davis established the following Lucas type congruence:
\begin{equation}\label{sdtp}
T_{p^{\alpha+1},l}(pn+s,pr+t)\equiv(-1)^t\binom{s}{t}T_{p^{\alpha},l}(n,r)\pmod{p},
\end{equation}
where $p$ is a prime, $\alpha\geq 1$, $n, r\geq 0$ and $0\leqslant s ,t\leqslant p-1$.
Now we may give a $q$-analogue of (\ref{sdtp}). For $b,c\geq 1$ with $b\mid c$, define
$$
T_{c,l}^{(b)}(n,r;q)=\frac{[l]_{q^{c}}!\Phi_c(q)^l}{[\floor{nb/c}]_{q^{c/b}}!}\sum_{k\equiv r\pmod{c}}
(-1)^kq^{\binom{k}{2}}\qbinom{n}{k}{q}\qbinom{(k-r)/c}{l}{q^{c}},
$$
where $[n]_q!=[n]_q[n-1]_q\cdots[1]_q$.
\begin{Thm} Let $b\geq 2$ and $n, r, s, t\geq 0$ be integers with $0\leqslant s ,t\leqslant b-1$.
Suppose that $c$ is a positive multiple of $b$. Then
\begin{equation}
\label{qlsdc}
T_{bc,l}^{(b)}(bn+s,br+t;q)\equiv(-1)^tq^{\binom{t}{2}}\qbinom{s}{t}{q}T_{c,l}^{(b)}(n,r;q^b)\pmod{\Phi_b(q)}.
\end{equation}
\end{Thm}
\begin{proof}
By the $q$-Lucas congruence (cf. \cite[Proposition 2.2]{D}), we have
\begin{equation}
\label{T2e1}
\qbinom{bn+s}{br+t}{q}\equiv\binom{n}{r}\qbinom{s}{t}{q}\pmod{\Phi_b(q)}.
\end{equation}
Since $q^b\equiv1\pmod{\Phi_b(q)}$,
(\ref{T2e1}) can be rewritten as
\begin{equation}
\label{T2e2}
\qbinom{bn+s}{br+t}{q}\equiv\qbinom{n}{r}{q^b}\qbinom{s}{t}{q}\pmod{\Phi_b(q)}.
\end{equation}
On the other hand, we have
\begin{equation}\label{qd}
(-1)^{bk+t}q^{\binom{bk+t}{2}}
\equiv
(-1)^{k+t}q^{\binom{t}{2}}\pmod{\Phi_b(q)}.
\end{equation}
In fact, since $q^{\binom{bk+t}{2}}=q^{bk(bk+2t-1)/2+\binom{t}{2}}$, (\ref{qd}) easily follows when $b$ is odd. And
if $b$ is even, then
$$
q^{b/2}=\frac{1-q^b}{1-q^{b/2}}-1\equiv-1\pmod{\Phi_b(q)}.
$$
Thus (\ref{qd}) is also valid for even $b$.
Since $b\mid c$, it is not difficult to see that $\Phi_{bc}(q)=\Phi_{c}(q^b)$.
Also, $[j]_{q^{c}}=(1-q^{jc})/(1-q^c)$ is prime to $\Phi_b(q)$ for any $j\geq 1$.
Hence,
\begin{align*}
&T_{bc,l}^{(b)}(bn+s,br+t;q)\\
=&\frac{[l]_{q^{bc}}!\Phi_{bc}(q)^l}{[\floor{(bn+s)/c}]_
{q^{(bc)/b}}!}\sum_{\substack{bk+t\equiv br+t\\\pmod{bc}}}
(-1)^{bk+t}q^{\binom{bk+t}{2}}\qbinom{bn+s}{bk+t}{q}\qbinom{((bk+t)-(br+t))/(bc)}{l}{q^{bc}}\\
\equiv&\frac{[l]_{(q^{b})^c}!\Phi_c(q^b)^l}{[\floor{nb/c}]_
{(q^{b})^{c/b}}!}\sum_{k\equiv r\pmod{c}}
(-1)^{k+t}q^{b\binom{k}{2}+\binom{t}{2}}\qbinom{n}{k}{q^b}\qbinom{s}{t}{q}\qbinom{(k-r)/c}{l}{q^{bc}}\\
=&(-1)^tq^{\binom{t}{2}}\qbinom{s}{t}{q}T_{c,l}^{(b)}(n,r;q^b)\pmod{\Phi_b(q)}.
\end{align*}
\end{proof}
Furthermore, define
$$
T_{c,l}^{(b)}(n,r;q,z)=\frac{[l]_{q^{c}}!\Phi_c(q)^l}{[\floor{nb/c}]_{q^{c/b}}!}\sum_{k\equiv r\pmod{c}}
(-1)^kz^kq^{\binom{k}{2}}\qbinom{n}{k}{q}\qbinom{(k-r)/c}{l}{q^{c}}.
$$
Then we also have
\begin{equation}
T_{bc,l}^{(b)}(bn+s,br+t;q,z)\equiv(-1)^tz^tq^{\binom{t}{2}}\qbinom{s}{t}{q}T_{c,l}^{(b)}(n,r;q^b,z^b)\pmod{\Phi_b(q)}.
\end{equation}

Below we consider the special case that $s=t=0$. We need the following $q$-analogue of the Wolstenholme-Ljunggren congruence.
\begin{Lem}
\begin{equation}\label{wlq}\qbinom{bn}{bm}q\bigg/\qbinom{n}{m}{q^b}\equiv
\big((-1)^{b-1}q^{\binom{b}{2}}\big)^{(n-m)m}+\frac{(b^2-1)nm(n-m)}{24}(1-q^{b})^2\pmod{\Phi_b(q)^3}.\end{equation}
\end{Lem}
\begin{proof} By Andrews' discussions in \cite{A}, we have
\begin{equation}\label{an}
\frac{(q^{jb+1};q)_{b-1}-(-1)^{j(b-1)}q^{j\binom{b}2}(q;q)_{b-1}}{(1-q^{(j+1)b})(1-q^{jb})}\equiv\frac{(b^2-1)b}{24}\pmod{\Phi_b(q)},
\end{equation}
though he only proved (\ref{an}) when $b$ is prime. Noting that $(q;q)_{b-1}\equiv b\pmod{\Phi_b(q)}$ and $1-q^{jb}\equiv j(1-q^b)\pmod{\Phi_b(q)^2}$, (\ref{an}) can be rewritten as
$$
\frac{(q^{jb+1};q)_{b-1}}{(q;q)_{b-1}}\equiv
(-1)^{j(b-1)}q^{j\binom{b}2}+\frac{(b^2-1)j(j+1)}{24}(1-q^{b})^2\pmod{\Phi_b(q)^3},
$$
It follows that
\begin{align*}
&\qbinom{bn}{bm}q\bigg/\qbinom{n}{m}{q^b}=\frac{\prod_{j=n-m}^{n-1}\big((q^{jb+1};q)_{b-1}/(q;q)_{b-1}\big)}{\prod_{j=0}^{m-1}\big((q^{jb+1};q)_{b-1}/(q;q)_{b-1}\big)}\\
\equiv&(-1)^{(b-1)(n-m)m}q^{\binom{b}{2}(n-m)m}\bigg(1+\frac{(b^2-1)nm(n-m)}{24}(1-q^{b})^2\bigg)\pmod{\Phi_b(q)^3}.\end{align*}
In view of (\ref{qd}), we get (\ref{wlq}).
\end{proof}
Thus,
\begin{align*}
&\frac{[l]_{q^{bc}}!\Phi_{bc}(q)^l}{[\floor{nb/c}]_
{q^{bc}}!}\sum_{bk\equiv br\pmod{bc}}
(-1)^{bk}z^{bk}q^{\binom{bk}{2}}\qbinom{bn}{bk}{q}\qbinom{(bk-br)/(bc)}{l}{q^{bc}}\\
\equiv&\frac{[l]_{(q^{b})^c}!\Phi_c(q^b)^l}{[\floor{nb/c}]_
{(q^{b})^{c/b}}!}\bigg((-1)^{(b-1)nr}\sum_{k\equiv r\pmod{c}}
(-1)^{k}z^{bk}q^{nk\binom{b}2+b\binom{k}2}\qbinom{n}{k}{q^b}\qbinom{(k-r)/c}{l}{q^{bc}}\\
&+\frac{(b^2-1)n}{24}\sum_{k\equiv r\pmod{c}}
(-1)^{k}z^{bk}(1-q^{kb})(1-q^{(n-k)b})\qbinom{n}{k}{q^b}\qbinom{(k-r)/c}{l}{q^{bc}}\bigg)\\
&\pmod{\Phi_b(q)^3}.
\end{align*}
That is,
\begin{Thm} Suppose that $b,c\geq 2$ and $b\mid c$. Then for $n,r\geq 0$, 
\begin{align}&\frac{1}{(1-q^b)^2}\cdot(T_{bc,l}^{(b)}(bn,br;q,z)-
(-1)^{(b-1)nr}T_{c,l}^{(b)}(n,r;q^b,z^bq^{n\binom{b}2}))\notag\\
\equiv&-\frac{\floor{(n-2)b/c}!}{\floor{nb/c}!}\cdot\frac{(b^2-1)n^2(n-1)}{24}\cdot z^bT_{c,l}^{(b)}(n-2,r-1;q^b,z^b)\pmod{\Phi_b(q)}.
\end{align}
In particular,
\begin{align}&T_{b^2,l}^{(b)}(bn,br;q,z)-
(-1)^{(b-1)nr}T_{b,l}^{(b)}(n,r;q^b,z^bq^{n\binom{b}2})\notag\\
\equiv&-\frac{(b^2-1)n}{24}\cdot z^b(1-q^b)^2T_{b,l}^{(b)}(n-2,r-1;q^b,z^b)\pmod{\Phi_b(q)^3}.
\end{align}

\end{Thm}

\section{From $q$-congruences to integer congruences}
\setcounter{equation}{0}
\setcounter{Thm}{0}
\setcounter{Lem}{0}
\setcounter{Cor}{0}

In \cite{SD}, Sun and Davis conjectured that
\begin{align}\label{sdc}
&\frac{p^l}{\floor{n/p}!}\sum_{k\equiv r\pmod{p}}(-1)^k\binom{pn}{pk}\bigg(\frac{k-r}{p}\bigg)^l\notag\\
\equiv&\frac{p^l}{\floor{n/p}!}\sum_{k\equiv r\pmod{p}}(-1)^k\binom{n}{k}\bigg(\frac{k-r}{p}\bigg)^l\pmod{p^3}
\end{align}
for prime $p\geq 5$. This conjecture was confirmed by Sun in \cite{S2}, with help of some arithmetical properties of the Stirling numbers of the second kind. 

Define
$$
S_{c,l}^{(b)}(n,r;q,z)=\frac{\Phi_c(q)^l}{[\floor{nb/c}]_{q^{c/b}}!}\sum_{k\equiv r\pmod{c}}
(-1)^kz^kq^{\binom{k}{2}}\qbinom{n}{k}{q}[(k-r)/c]_{q^{c}}^l.
$$
Using the similar discussions as above, we also can obtain that
\begin{equation}
S_{bc,l}^{(b)}(bn+s,br+t;q,z)\equiv(-1)^tz^tq^{\binom{t}{2}}\qbinom{s}{t}{q}S_{c,l}^{(b)}(n,r;q^b,z^b)\pmod{\Phi_b(q)},
\end{equation}
and
\begin{align}&\frac{1}{(1-q^b)^2}\cdot(S_{bc,l}^{(b)}(bn,br;q,z)-
(-1)^{(b-1)nr}S_{c,l}^{(b)}(n,r;q^b,z^bq^{n\binom{b}2}))\notag\\
\equiv&-\frac{\floor{(n-2)b/c}!}{\floor{nb/c}!}\cdot\frac{(b^2-1)n^2(n-1)}{24}\cdot z^bS_{c,l}^{(b)}(n-2,r-1;q^b,z^b)\pmod{\Phi_b(q)}.
\end{align}
In particular, for prime $p\geq 5$,
\begin{align}\label{sdcq}
&\frac{[p]_{q^{p^{\alpha}}}^l}{[\floor{n/p^\alpha}]_{q^{p^{\alpha}}}!}\sum_{k\equiv r\pmod{p^\alpha}}
(-1)^kq^{\binom{kp}{2}}\qbinom{pn}{pk}{q}[(k-r)/p^\alpha]_{q^{p^\alpha}}^l\notag\\
\equiv&\frac{[p]_{q^{p^{\alpha}}}^l}{[\floor{n/p^\alpha}]_{q^{p^{\alpha}}}!}\sum_{k\equiv r\pmod{p^\alpha}}
(-1)^kq^{nk\binom{p}2+p\binom{k}{2}}\qbinom{n}{k}{q^p}[(k-r)/p^\alpha]_{q^{p^\alpha}}^l\notag\\
+&\frac{(p^2-1)n[p]_{q^{p^{\alpha}}}^l}{24[\floor{n/p^\alpha}]_{q^{p^{\alpha}}}!}\sum_{k\equiv r\pmod{p^\alpha}}
(-1)^k(1-q^{kp})(1-q^{(n-k)p})\qbinom{n}{k}{q^p}[(k-r)/p^\alpha]_{q^{p^\alpha}}^l\notag\\
&\pmod{[p]_q^3}.
\end{align}

However, one might doubt whether (\ref{sdcq}) surely implies (\ref{sdc}), since neither side of (\ref{sdcq}) is a polynomial in $q$. So we need to give an explanation how to deduce (\ref{sdc}) from (\ref{sdcq}) by substituting $q=1$.

Let $L(q)$ and $R(q)$ denote the left side and the right side of (\ref{sdcq}) respectively. Let
$$
F(q)=\frac{[\floor{n/p^\alpha}]_{q^{p^{\alpha}}}!}{\prod_{j\geq \alpha+1}[p]_{q^{p^{j-1}}}^{\floor{{n}/{p^{j}}}}}.
$$
Then clearly $F(q)\in\Z[q]$ in view of (\ref{nd}). And from Corollary \ref{c1}, we also know that $F(q)L(q), F(q)R(q)\in\Z[q]$.
Hence there exists a polynomial $H(q)\in\Z[q]$ such that
$$
F(q)L(q)-F(q)R(q)=[p]_q^3H(q).
$$
Substituting $q=1$, we get $$F(1)L(1)\equiv F(1)R(1)\pmod{p^3}.$$ But by (\ref{ndd}), $$F(1)=\frac{\floor{n/p^\alpha}!}{p^{\sum_{j\geq \alpha+1}\floor{{n}/{p^{j}}}}}$$ is not divisible by $p$. Thus (\ref{sdc}) is concluded.
\begin{Ack}
I am grateful to Professor Zhi-Wei Sun for his helpful suggestions on this paper.
\end{Ack}

\end{document}